\documentclass[11pt]{amsart}

\usepackage{amssymb,mathrsfs,enumitem,amsmath,amsthm,bbold}
\usepackage[mathscr]{euscript}
\usepackage{color}
\usepackage{epigraph}
\usepackage[all]{xy}
\definecolor{Chocolat}{rgb}{0.36, 0.2, 0.09}
\definecolor{BleuTresFonce}{rgb}{0.215, 0.215, 0.36}
\definecolor{EgyptianBlue}{rgb}{0.06, 0.2, 0.65}

\usepackage[OT2,OT1]{fontenc}

\usepackage[backend=bibtex,
    sorting=anyt,
    isbn=false,
    url=false,
    doi=false,
    maxbibnames=100
    ]{biblatex}
\usepackage{fourier}

\usepackage[colorlinks,final,hyperindex]{hyperref}
\hypersetup{citecolor=BleuTresFonce, urlcolor=EgyptianBlue, linkcolor=Chocolat}

\newtheorem{theorem}{Theorem}[section]
\newtheorem{corollary}[theorem]{Corollary}

\newtheorem{proposition}[theorem]{Proposition}

\newtheorem{conjecture}[theorem]{Conjecture}

\theoremstyle{definition}

\newtheorem{remark}[theorem]{Remark}

\newcommand{\ac}{\scriptstyle \text{\rm !`}}

\DeclareMathAlphabet{\pazocal}{OMS}{zplm}{m}{n}

\def\calA{\pazocal{A}}
\def\calB{\pazocal{B}}

\def\calF{\pazocal{F}}
\def\calG{\pazocal{G}}

\def\calK{\pazocal{K}}
\def\calL{\pazocal{L}}

\def\calO{\pazocal{O}}

\def\calR{\pazocal{R}}

\def\calT{\pazocal{T}}

\DeclareMathOperator{\Lie}{Lie}

\DeclareMathOperator{\Jord}{Jord}

\DeclareMathOperator{\Vect}{{\ensuremath\mathsf{Vect}}}

\DeclareMathAlphabet{\mathbbold}{U}{bbold}{m}{n}

\def\k{\mathbbold{k}}

\begin{document}

\title[On the conjecture of Kashuba and Mathieu about free Jordan algebras]{On the conjecture of Kashuba and Mathieu\\ about free Jordan algebras}

\author{Vladimir Dotsenko}
\address{Institut de Recherche Math\'ematique Avanc\'ee, UMR 7501, Universit\'e de Strasbourg et CNRS, 7 rue Ren\'e-Descartes, 67000 Strasbourg, France}

\email{vdotsenko@unistra.fr}

\author{Irvin Roy Hentzel}
\address{Department of Mathematics, Iowa State University, Ames, Iowa 50011, United States of America}

\email{hentzel@iastate.edu}

\begin{abstract}
Kashuba and Mathieu \cite{MR4235202} proposed a conjecture on vanishing of Lie algebra homology, implying a description of the $GL_d$-module structure of the free $d$-generated Jordan algebra. Their conjecture relies on a functorial version of the Tits--Kantor--Koecher construction that builds Lie algebras out of Jordan algebras. In this note, we summarize new intricate computational data concerning free Jordan algebras and explain why, despite a lot of overwhelmingly positive evidence, the conjecture of Kashuba and Mathieu is not true. 
\end{abstract}

\maketitle

\section{Introduction}

A Jordan algebra is a nonassociative commutative algebra in which the identity
 \[
a(a^2b)=a^2(ab)     
 \] 
holds; in other words, for any element $a$, multiplications by $a$ and by $a^2$ commute. The original motivation for Jordan algebras came from a wish to introduce a new algebraic formalism for quantum mechanics \cite{MR1503141}; in particular, one key guiding observation was that in any associative algebra with involution, the subspace of self-adjoint elements is not closed under product but is closed under the symmetrized product $ab:=\frac12(a\cdot b+b\cdot a)$. That latter product is known to satisfy the defining identity of Jordan algebras, but in fact satisfies other identities too \cite{MR0186708}. 
  
A lot of structural results on Jordan algebras are known \cite{MR668355}, yet free Jordan algebras are not really well understood. They have many unexpected properties: for instance, free Jordan algebras on at least $3$ generators contain zero divisors~\cite{MR566779} and free Jordan algebras on at least $32$ generators contain nilpotent elements~\cite{MR788339}. One important aspect of the theory of Jordan algebras is their relationship to Lie algebras via a beautiful construction of Tits \cite{MR0146231}, later generalized by Kantor \cite{MR0175941} and Koecher \cite{MR214631}. Several years ago, Kashuba and Mathieu~\cite{MR4235202} used a functorial version of the Tits--Kantor--Koecher construction \cite{MR1752782,MR3169564} to propose a conjecture that would shed light on free Jordan algebras. Concretely, they conjectured that the homology of any Lie algebra obtained from a \emph{free} Jordan algebra by a functorial version of the Tits--Kantor--Koecher construction, if viewed as an $\mathfrak{sl}_2$-module, contains no trivial or adjoint isotypic components in homological degrees greater than one. This conjecture would imply a description of the $GL_d$-module structure of the free $d$-generated Jordan algebra for each $d\in\mathbb{N}$. A version of that conjecture for free Jordan superalgebras was later proposed by Shang \cite{MR4853483}.

In this note, we use the algorithms of the second author \cite{MR435159,MR463251} to obtain new computational data on free Jordan algebras that overwhelmingly supports this conjecture (it predicts correctly the dimension of the space of elements of degree $n$ in the free Jordan algebra on two generators for $n\le 18$, the dimension of the space of elements of degree $n$ in the free Jordan algebra on three generators for $n\le 12$, and the dimension, as well as the $GL_k$-module structure, of the space of elements of degree $n$ in the free Jordan algebra on $k\ge 4$ generators for $n\le 10$), and, somewhat disappointingly in the light of that data, we disprove the conjecture. In fact, the conjecture fails already for the free Jordan algebra on two generators, predicting the wrong dimension of the space of elements of degree~$19$ in that algebra; this rather surprising discovery followed the discovery of the first author \cite{dotsenko2025conjectureshangfreealternative} that the case of the conjecture of Shang~\cite{shang2025allisonbenkartgaofunctorcyclicityfree} about free alternative algebras, inspired by the conjecture of Kashuba and Mathieu, also fails for two-generated algebras. 
Recall that the results of Cohn \cite{MR60496} and Shirshov \cite{MR75936} allow one to describe the free Jordan algebra on two generators in a completely explicit way, so the prediction can be checked in any degree. It is perhaps unconventional for a mathematics paper explaining that a certain conjecture is false to also offer data in favour of that conjecture. There are two reasons to do so. First, one may say that this choice is informed by the celebrated Feynman's address~\cite{Feyn}, and specifically the part 
\begin{quote}
In summary, the idea is to try to give all of the information to help others to judge the value of your contribution; not just the information that leads to judgment in one particular direction or another.
\end{quote}
Second, it is reasonable to say that now that the conjecture of Kashuba and Mathieu is disproved, the subject of free Jordan algebras is as mysterious as ever. Thus, unlike the paper \cite{dotsenko2025conjectureshangfreealternative}, which was partially written in parallel to the present one and which focuses on several theoretical phenomena and mostly relies on previously available computational data, the present paper additionally presents results of numerous new computations, which may eventually prove instrumental in further studies of free Jordan algebras. 

This paper is organized as follows. In Section \ref{sec:conjecture}, we recall the necessary definitions and the main conjecture. In Section \ref{sec:operads}, we discuss an operadic viewpoint of the functorial Tits--Kantor--Koecher construction and of the conjecture of Kashuba and Mathieu, as well as of its superalgebra version. In Section \ref{sec:counterexamples}, we present new computational results for dimensions and module structures of subspaces of free Jordan algebras, and explain why the conjecture is false for free Jordan algebras and superalgebras with at least two even generators. 

\subsection{Acknowledgements} We thank Frederic Chapoton, Iryna Kashuba, Olivier Mathieu, and Ivan Shestakov for numerous useful discussions. Work of V. D. was supported by the ANR (French national research agency, project ANR-20-CE40-0016) and by Institut Universitaire de France. Some of the computations presented here were performed during the visit of V. D. to the Shenzhen International Center for Mathematics; he thanks this institution for hospitality and excellent working conditions. 

\section{The Tits--Allison--Gao functor and the conjecture of Kashuba and Mathieu}\label{sec:conjecture}

We refer the reader to the monograph \cite{MR2014924} for extensive information on Jordan algebras. Throughout this paper, we use the following notation in any Jordan algebra $J$: $L_g(f)=g\cdot f$, and denote by $\k$ the ground field, which we assume to be of zero characteristic. Recall that for any Jordan algebra $J$ and any $a,b\in J$, the operator $D_{a,b}\in \mathrm{End}(J)$ defined by the formula
 \[
D_{a,b}:=[L_a,L_b] 
 \]
is a derivation of $J$. Since the Jordan axiom means $[L_a,L_{a^2}]=0$, it is easy to see that these operators satisfy the following identities:
\begin{gather}
D_{a,b}+D_{b,a}=0,\label{eq:antisym}\\
D_{ab,c}+D_{bc,a}+D_{ca,b}=0,\label{eq:cyclic}\\
[D,D_{a,b}]=D_{D(a),b}+D_{a,D(b)} \text{ for all } D\in \mathrm{Der}(J).
\end{gather} 
Derivations of this form are called \emph{inner}; they arise as a particular case of the general theory of inner derivations of nonassociative algebras \cite[Sec.~II.3]{MR210757}. The last property shows that inner derivations $\mathrm{Inner}(J)$ form an ideal of the Lie algebra $\mathrm{Der}(J)$ of all derivations of $J$. The celebrated result of Tits \cite{MR0146231} later generalized by Kantor \cite{MR0175941} and Koecher \cite{MR214631}, states that for a Jordan algebra $J$, the vector space
$\mathfrak{sl}_2\otimes J\oplus \mathrm{Inner}(J)$ 
has a Lie algebra structure given by
\begin{gather*}
[x\otimes a,y\otimes b]=
[x,y]\otimes ab+\frac12\mathrm{tr}(xy)D_{a,b},\\
[D_{a,b},x\otimes c]=x\otimes D_{a,b}(c),\\
[D_{a,b},D_{c,d}]=D_{D_{a,b}(c),d}+D_{c,D_{a,b}(d)}.
\end{gather*}
 However, inner derivations are not functorial, and hence this construction does not give a functor from the category of Jordan algebras to the category of Lie algebras. However, this is remedied if one passes to the universal central extension of this latter algebra, described as follows \cite[Th.~4.13]{MR1752782}. Guided by Equations \eqref{eq:antisym} and \eqref{eq:cyclic}, one associates to the given Jordan algebra $J$ the vector space 
 \[
\calB(J):=\Lambda^2(J)/(ab\wedge c+bc\wedge a+ca\wedge b\colon a,b,c\in J)    
 \]
which one can think of as a functorial version of the space of $\mathrm{Inner}(J)$. Furthermore, one defines on the vector space 
 \[
\mathsf{TAG}(J):=\mathfrak{sl}_2\otimes J\oplus\calB(J)   
 \]
a skew-symmetric bracket by
\begin{gather}
\label{eq:TAG1}
[x\otimes a,y\otimes b]=
[x,y]\otimes ab+\frac12\mathrm{tr}(xy) a\wedge b,\\
[a\wedge b,x\otimes c]=x\otimes D_{a,b}(c),\label{eq:TAG2}\\
[a\wedge b,c\wedge d]=D_{a,b}(c)\wedge d+c\wedge D_{a,b}(d).\label{eq:TAG3}
\end{gather}
One can show that these formulas make $\mathsf{TAG}(A)$ into a Lie algebra that one calls the \emph{Tits--Allison--Gao Lie algebra} associated to $J$. Clearly, the construction $\mathsf{TAG}(A)$ is functorial, and produces a Lie algebra in the category of $\mathfrak{sl}_2$-modules having only trivial and adjoint components. Tits proved in \cite{MR0146231} that for each such Lie algebra $\mathfrak{g}$, the multiplicity of the adjoint component has a Jordan algebra structure, which we shall call the \emph{Tits functor}. It is known \cite{MR3169564,MR4235202} that $\mathsf{TAG}$ is the left adjoint of the Tits functor. In \cite{MR4235202}, the following conjecture is proposed. 

\begin{conjecture}[{\cite[Conjecture 2]{MR4235202}}]\label{conj:conj1}
Let $\Jord(V)$ denote the free Jordan algebra generated by a finite-dimensional vector space $V$. The $\mathfrak{sl}_2$-module 
 \[
H_k(\mathsf{TAG}(\Jord(V)),\k)     
 \]
has no trivial or adjoint component for $k>1$. 
\end{conjecture}

It is shown in \cite[Lemma 1]{MR4235202} that there exist unique elements $a(V)$ and $b(V)$ in the augmentation ideal of the Grothendieck ring of $GL(V)$ for which we have the following equalities in the Grothendieck ring of $GL(V)\times PSL_2$:
\begin{gather*}
[\lambda(a(V)[L(2)]+b(V)[L(0)])\colon [L(0)]]=[\k]\\
[\lambda(a(V)[L(2)]+b(V)[L(0)])\colon [L(2)]]=-[V].
\end{gather*}
Here $\lambda$ is the homological $\lambda$-operation satisfying $\lambda(x+y)=\lambda(x)\lambda(y)$ and given on each effective class $x=[U]$ by 
 \[
\lambda(x)=\sum_{k\ge 0}(-1)^k[\Lambda^k(U)].     
 \]

According to \cite[Cor.~1]{MR4235202}, Conjecture \ref{conj:conj1} implies the following conjecture on character formulas for $\Jord(V)$ and $\calB(\Jord(V))$ in the Grothendieck ring of $GL(V)$. 

\begin{conjecture}[{\cite[Conjecture 1]{MR4235202}}]\label{conj:conj2}
Let $\Jord(V)$ denote the free Jordan algebra generated by a finite-dimensional vector space $V$. In the Grothendieck ring of $GL(V)$, we have
 \[
[\Jord(V)]=a(V), \quad [\calB(\Jord(V))=b(V)].     
 \]
\end{conjecture}

The proof of \cite[Lemma 1]{MR4235202} gives an explicit recursive procedure for computing $a(V)$ and $b(V)$ which we implemented in \texttt{SageMath} to compute the predictions of Conjecture \ref{conj:conj2}.  

\section{Operads, algebras and superalgebras}\label{sec:operads}

This section, setting an operad theory context for the conjecture of Kashuba and Mathieu, is closely related to the corresponding section of \cite{dotsenko2025conjectureshangfreealternative}. Let us offer some basic insight to the operad theory for a ring theorist reader, referring to \cite[Sec.~2.3]{MR4675074} for further details and to the monograph \cite{MR2954392} for systematic information on operads. Recall that, to a sequence $\{K(n)\}_{n\ge 0}$, where each $K(n)$ is a right module over the symmetric group $S_n$, one can associate a functor $\calK\colon \Vect\to\Vect$ given by the formula
 \[
\calK(V):= \bigoplus_{n\ge 0}K(n)\otimes_{\k S_n}V^{\otimes n} .   
 \]
Functors like that are called \emph{analytic functors} \cite{MR633783}: their value on $V$ is a ``categorified Taylor series'' with the ``iterated derivatives'' $K(n)$ (one can view the tensoring over $\k S_n$ as the categorical division by $n!$). An important class of analytic functors come from free algebras of various kinds. For instance, if we denote by $\Jord(n)$ the subspace of the free Jordan algebra $\Jord(x_1,\ldots,x_n)$ consisting of all elements of degree exactly one in each generator $x_1,\ldots,x_n$, this vector space has a natural right $S_n$-action (by permutations of the generators $x_1,\ldots,x_n$), and hence the collection of all these spaces 
 \[
\Jord:=\{\Jord(n)\}_{n\ge 1}
 \]
gives rise to an analytic functor. It is easy to see that $\Jord(V)$ can be naturally identified with the free Jordan algebra generated by $V$: indeed, the term 
 \[
 \Jord(n)\otimes_{\k S_n}V^{\otimes n}    
 \]
may be viewed as the result of ``evaluation'' of multilinear Jordan elements on all possible $n$-tuples of elements of $V$; the tensor product over the symmetric group ensures that this evaluation is consistent with simultaneously renumbering elements of the $n$-tuple and permuting the entries of the multilinear Jordan element. Furthermore, the natural map 
 \[
\Jord(\Jord(V))\to\Jord(V)     
 \]
which simply forgets the layered structure of Jordan operations on the left, considered together with the obvious inclusion 
$V\hookrightarrow\Jord(V)$ makes the analytic functor $\Jord$ into a monad; monads like that are called \emph{operads}.

Note that over a field of characteristic zero, every identity is equivalent to a multilinear identity, and in particular the Jordan identity is equivalent to the multilinear identity
 \[
((ab)c)d+
((bd)c)a+
((ad)c)b=
(ab)(cd)+
(ac)(bd)+
(ad)(bc).
 \] 
This allows the reader fluent in the language of operads to present the operad $\Jord$ by means of generators and relations.

For our purposes, it will be important that analytic functors form a symmetric monoidal category with respect to the so called \emph{Cauchy product}; if $\calF=\{F(n)\}_{n\ge 0}$ and $\calG=\{G(n)\}_{n\ge 0}$ are two analytic functors, we may define a new analytic functor $\calF\otimes\calG$ whose $n$-th component is given by  
 \[
\bigoplus_{k+l=n}\mathrm{Ind}_{S_k\times S_l}^{S_n}(F(k)\otimes G(l)).     
 \]
It is easy to check that this product is associative and admits symmetry isomorphisms $\calF\otimes\calG\to \calG\otimes\calF$ satisfying all necessary axioms of a symmetric monoidal category. In particular, one can talk about Lie algebras in this category, which historically are called \emph{twisted Lie algebras} \cite{MR513566}. Concretely, a twisted Lie algebra $\mathfrak{g}$ may be viewed as a Lie algebra of the form $\mathfrak{g}=\bigoplus_{n\ge 0}\mathfrak{g}(n)$, 
where each $\mathfrak{g}(n)$ is a right $\k S_n$-module, and the Lie bracket maps $\mathfrak{g}(n)\otimes \mathfrak{g}(m)$ to $\mathfrak{g}(n+m)$ and is $S_n\times S_m$-equivariant. Clearly, if $\mathfrak{g}$ is a twisted Lie algebra and $V$ is a vector space, then we may view $\mathfrak{g}$ as an analytic functor and obtain the vector space $\mathfrak{g}(V)$; the twisted Lie algebra structure of $\mathfrak{g}$ induces an honest Lie algebra structure on $\mathfrak{g}(V)$. In particular, any Lie algebra $L$ gives rise to a twisted Lie algebra, if one considers the ``constant'' analytic functor 
\[
\mathsf{1}_L(n)=
\begin{cases}
L, \quad n=0,\\
0, \quad n>0.
\end{cases}     
 \]
Another viewpoint on this same definition is that the analytic functor $\mathfrak{g}$ admits an ``action'' of the analytic functor of the Lie operad, that is, a natural transformation $\Lie\circ\mathfrak{g}\to\mathfrak{g}$: indeed, a Lie algebra structure on $\mathfrak{g}(V)$ can be viewed as a map $\Lie(\mathfrak{g}(V))\to \mathfrak{g}(V)$ that is consistent with the operad structure of $\Lie$, and this is natural in $V$. In general, for an operad $\calO$, the notion of a twisted $\calO$-algebra is equivalent to that of a left module over the operad $\calO$. One can also define right modules over an operad, a right $\calO$-module is an analytic functor $\calK$ together with a natural transformation $\calK\circ\calO\to \calK$. Right modules are not algebras, but rather constructions depending on algebras in a way that is stable under algebra endomorphisms, for example, the commutator quotient $A/[A,A]$ of an associative algebra (in the context of PI-algebras, the notion corresponding to that of an operadic right module is known as a $T$-space).

The following result lifts the $\mathsf{TAG}$ construction to the level of twisted Lie algebras, and sheds new light on the functorial properties of that construction. Recall that if $\calR$ is a right $\calO$-module and $\calL$ is a left $\calO$-module, we can form, by analogy with the tensor product of modules over a ring, the relative composition product $\calR\circ_\calO\calL$, see \cite[Sec.~5.1.5]{MR2494775}.

\begin{proposition}\leavevmode
\begin{enumerate}
\item There exists a twisted Lie algebra $\calT\calA\calG$ such that we have a Lie algebra isomorphism 
 \[
 \mathsf{TAG}(\Jord(V))\cong \calT\calA\calG(V).    
 \]
\item The twisted Lie algebra $\calT\calA\calG$ is a right module over the Jordan operad. Moreover, for every Jordan algebra
$J$, we have a natural isomorphism of left $\Lie$-modules
 \[
\mathsf{1}_{\mathsf{TAG}(J)}\cong \calT\calA\calG\circ_{\Jord} \mathsf{1}_A.     
 \]
\end{enumerate} 
\end{proposition}

\begin{proof}
Let $V$ be a finite-dimensional vector space. We know that the free Jordan algebra $\Jord(V)$ can be described as the value of the analytic functor of the Jordan operad on the vector space $V$:
 \[
\Jord(V)\cong\bigoplus_{n\ge 1}\Jord(n)\otimes_{\k S_n}V^{\otimes n}.     
 \]
Moreover, the functor 
 \[
\calB(\Jord(V))\cong \Lambda^2(\Jord(V))/(ab\wedge c+bc\wedge a+ca\wedge b \colon a,b,c\in \Jord(V))     
 \]
is also analytic. Indeed, we may take the obvious functorial version of the definition of $\calB$, looking, for each $n$, at the elements of $\calB(\Jord(x_1,\ldots,x_n))$ consisting of all elements of degree exactly one in each generator $x_1,\ldots,x_n$. Equivalently, one may define 
 \[
B:=\calB(\Jord)= \Lambda^2(\Jord)/(ab\wedge c+bc\wedge a+ca\wedge b \colon a,b,c\in \Jord) ,   
 \]
where the exterior power is now taken with respect to the Cauchy product. We can now define a twisted Lie algebra structure on the analytic functor
 \[
\calT\calA\calG=\mathfrak{sl}_2\otimes\Jord\oplus \calB(\Jord)     
 \]
by the same formulas \eqref{eq:TAG1}--\eqref{eq:TAG3}. We conclude that 
 \[
\mathsf{TAG}(\Jord(V))=\mathfrak{sl}_2\otimes\Jord(V)\oplus \calB(\Jord(V))\cong (\mathfrak{sl}_2\otimes\Jord\oplus \calB(\Jord))(V)   
 \]
is an analytic functor and a twisted Lie algebra (one can even view it as the Tits--Allison--Gao construction applied to the twisted Jordan algebra $\Jord$), and that $\mathsf{TAG}(\Jord(V))\cong \calT\calA\calG(V)$. 

Furthermore, the operad $\Jord$ is tautologically a right $\Jord$-module, and the analytic functor $\calB(\Jord)$ is a right $\Jord$-module since $\calB(\Jord(V)))$ is manifestly stable under all endomorphisms of $\Jord(V)$, so the analytic functor $\calT\calA\calG$ is a right $\Jord$-module. Moreover,  the Lie algebra structure on $\mathsf{TAG}(\Jord(V))$ clearly commutes with all endomorphisms of $\Jord(V)$, so the twisted Lie algebra structure on $\calT\calA\calG$ commutes with the right $\Jord$-module action. It remains to notice that
\begin{gather*}
\Jord \circ_{\Jord} \mathsf{1}_J\cong \mathsf{1}_J,\\
\calB(\Jord)\circ_{\Jord} \mathsf{1}_J\cong \mathsf{1}_{\calB(J)},
\end{gather*}
where the first isomorphism is obvious, and the second follows from the fact that $\calB(\Jord)$ is defined by the same formula as $\calB(J)$, and the relative composition product ``computes'' all the products in $J$ after the evaluation of the analytic functor $\calB(\Jord)(J)$ by taking the corresponding coequalizer. Since the twisted Lie algebra structure on $\calT\calA\calG$ commutes with the right $\Jord$-module action, the induced Lie algebra structure on the relative composite product matches the Tits--Allison--Gao Lie algebra structure on the vector space $\mathsf{TAG}(J)$.
\end{proof}

Let us also outline a functorial viewpoint on Lie algebra homology; for simplicity, we shall focus on the homology with trivial coefficients. In classical textbooks, one would often find the either the definition of the homology as a derived functor
$H_\bullet(L,\k):=\mathrm{Tor}_\bullet^{U(L)}(\k,\k)$,  
or a definition via the Chevalley--Eilenberg complex  
 \[
H_\bullet(L,\k):=H_\bullet(\Lambda(L),d),     
 \]
where in $\Lambda(L)$ we place $\Lambda^k(L)$ in the homological degree $k$, and the differential $d$ is given by the formula
 \[
d(x_1\wedge\cdots\wedge x_k)=\sum_{1\le i< j\le k}(-1)^{i+j-1}[x_i,x_j]\wedge x_1\wedge\cdots\wedge \hat{x}_{i}\cdots\wedge \hat{x}_{j}\wedge\cdots\wedge x_k,     
 \] 
where the notation $\hat{x}_i$ means that this factor must be omitted. While these definitions are completely unambiguous, they disguise one conceptual aspect of the definition. The operad of Lie algebras is Koszul, and its Koszul dual is the operad of commutative associative algebras. This allows us to interpret the homology of a Lie algebra $L$ as the homology of its \emph{bar construction} \cite{MR2954392}. In our particular case, the space of exterior forms $\Lambda(L)$ is really a disguise of $S^c(sL)$, the cofree conilpotent cocommutative coassociative coalgebra on $sL$ (the vector space $L$ homologically shifted by one using an odd element $s$), and $d$ is the unique coderivation of $S^c(sL)$ extending the map 
 \[
S^c(sL)\twoheadrightarrow S^2(sL)\to sL    
 \]
made of the projection onto $S^c(sL)\twoheadrightarrow S^2(sL)$ and the map $S^2(sL)\to sL$ corresponding to the Lie bracket (a skew-symmetric bilinear map $V\times V\to V$ is the same as symmetric bilinear map $sV\times sV\to sV$); note that, similarly how derivations of free algebras are determined by restriction to generators, coderivations of cofree conilpotent coalgebras are determined by corestriction to cogenerators. With that in mind, we obtain the following result.

\begin{proposition}\label{prop:Kunneth}
We have 
 \[
H_\bullet(\mathsf{TAG}(V),\k)\cong  H_\bullet(\calT\calA\calG,\k)(V).    
 \]
\end{proposition}

\begin{proof}
The underlying vector space of the Chevalley--Eilenberg complex of the Lie algebra $\mathsf{TAG}(V)$ is, as we already indicated, 
 \[
(S^c(s\mathsf{TAG}(V)),d),   
 \]
and, if we note that $S^c(sV)\cong s\Lie^{\ac}(V)$, where $\Lie^{\ac}$ is the Koszul dual cooperad of the Lie operad \cite[Sec.~7.2]{MR2954392}, the differential $d$ comes from the Koszul twisting cochain $\kappa\colon\Lie^{\ac}\to\Lie$. Thus, computing the differential (and its homology) commutes with the evaluation of analytic functors on $V$.
\end{proof}

\begin{corollary}\label{cor:SchurWeyl}
Conjecture \ref{conj:conj1} holds for all choices of the vector space of generators $V$ if and only if the $\mathfrak{sl}_2$-module 
$H_k(\calT\calA\calG,\k)$ 
has no trivial or adjoint component for $k>1$. 
\end{corollary}

\begin{proof}
Proposition \ref{prop:Kunneth} implies the ``if'' part of the statement. The ``only if'' part, that is the assertion that vanishing of the trivial and the adjoint component of $H_k(\mathsf{TAG}(V),\k)$ for all $V$ implies the same for $H_k(\calT\calA\calG,\k)$, follows from the fact that the ``Taylor coefficients'' of an analytic functor can be uniquely reconstructed from its values using the Schur--Weyl duality \cite{MR1153249}.
\end{proof}

In \cite{MR4853483}, a generalization of the Kashuba--Mathieu conjecture for Jordan superalgebras is proposed. From the operadic point of view, this generalization comes at no cost: over the years, the first author of this note has been advertising the viewpoint that superalgebras over a given operad are merely algebras over the same operad in a larger symmetric monoidal category \cite{DotRev}, and hence various results that can be stated and proved in terms of the corresponding operad (e.g., the Poincaré--Birkhoff--Witt type theorems \cite{MR4300233} or the Nielsen--Schreier property \cite{MR4675074}) are automatically true for superalgebras if already proved for algebras. To give a yet another example of this approach, let us record the following result.

\begin{proposition}\label{prop:super}
Suppose that Conjecture \ref{conj:conj1} is true for all free Jordan algebras. Then it is true for all free Jordan superalgebras. 
\end{proposition}

\begin{proof}
The free Jordan superalgebra $\Jord(V_0\mid V_1)$ generated by the superspace $V_0\mid V_1$ can also be written as a value of the same analytic functor:
 \[
\Jord(V_0\mid V_1)\cong  \Jord(n)\otimes_{\k S_n}(V_0\mid V_1)^{\otimes n},   
 \]
where the latter formula uses the symmetric monoidal structure on the category of $\mathbb{Z}/2\mathbb{Z}$-graded vector spaces given by the Koszul sign rule 
 \[
\sigma(u\otimes v)=(-1)^{|u||v|}v\otimes u.     
 \]
The Lie superalgebra homology also admits a description as the homology of $(S^c(sL),d)$, where $S^c(sL)$ now denotes the cofree conilpotent cocommutative coassociative co-superalgebra generated by $sL$ (which is nothing but the cofree conilpotent cocommutative coassociative \emph{coalgebra} in the symmetric monoidal category of $\mathbb{Z}/2\mathbb{Z}$-graded vector spaces), and $d$ is the unique coderivation of $S^c(sL)$ extending the map 
 \[
S^c(sL)\twoheadrightarrow S^2(sL)\to sL    
 \]
made of the projection onto $S^c(sL)\twoheadrightarrow S^2(sL)$ and the map $S^2(sL)\to sL$ corresponding to the Lie superalgebra structure. 
The obvious analogue of Proposition~\ref{prop:Kunneth} holds, so 
 \[
H_\bullet(\mathsf{TAG}(\Jord(V_0\mid V_1)),\k)\cong H_\bullet(\calT\calA\calG,\k)(V_0\mid V_1),     
 \]
and the claim follows from Corollary \ref{cor:SchurWeyl}. 
\end{proof}

\section{Examples and counterexamples to the conjecture}\label{sec:counterexamples}

\subsection{The \texorpdfstring{$\Jord(\varnothing \mid x)$}{J01} case}

In \cite{MR4853483}, Conjecture \ref{conj:conj2} is checked for the free Jordan superalgebra on one odd generator $x$. Let us establish a stronger result.

\begin{proposition}\label{prop:Jord01}
Conjecture \ref{conj:conj1} holds for the free Jordan superalgebra $\Jord(\varnothing \mid x)$. 
\end{proposition}

\begin{proof}
In this case, the corresponding Lie superalgebra $TAG(\Jord(\varnothing \mid x))$ is a very small finite-dimensional Lie superalgebra
 \[
\mathfrak{sl}_2\otimes \k x\oplus \k x\cdot x,     
 \]
with the only nonzero brackets $[ax,bx]=K(a,b)x\cdot x$ for $a,b\in\mathfrak{sl}_2$; here $K$ is the Killing form of $\mathfrak{sl}_2$. This Lie superalgebra can be viewed as the Heisenberg algebra of the odd vector space $\Pi\mathfrak{sl}_2$, on which the Killing form is a super-skew\-symmetric bilinear form. The cohomology of all Heisenberg superalgebras with trivial coefficients is computed in~\cite{MR3605960}. In the form suitable for our purposes, case $m=0$ of \cite[Th.~4.1]{MR3605960} implies that we have an algebra isomorphism
 \[
H^\bullet(TAG(\Jord(\varnothing \mid x)),\k)\cong S(\mathfrak{sl}_2^*)/(K),
 \]
where the Killing form $K$ is naturally identified with an element of $S^2(\mathfrak{sl}_2^*)$; moreover, under this isomorphism, the homological degree corresponds to the degree of polynomials. Note that as an $\mathfrak{sl}_2$-module, we have 
 \[
S(\mathfrak{sl}_2^*)/(K)\cong \bigoplus_{p\ge 0}L(2p),     
 \]
and we have
 \[
H^p(TAG(\Jord(\varnothing \mid x)),\k) \cong L(2p),  
 \]
meaning that we have an analogue of the Garland--Lepowsky theorem \cite{MR414645} in this case. In particular, $H^p$ does not contain trivial or adjoint components for $p>1$, proving the statement.
\end{proof}

\subsection{The case of degrees not exceeding \texorpdfstring{$10$}{10}}

The main result of this section is the following result, which we regarded as very strong evidence in favor of Conjecture~\ref{conj:conj1}. 

\begin{proposition}
Conjecture \ref{conj:conj2} is true for all Jordan algebras and superalgebras in degrees not exceeding $10$.
\end{proposition}

\begin{proof}
Using \texttt{SageMath} \cite{sagemath}, we computed the prediction of Conjecture \ref{conj:conj2} about the $S_n$-module structure of $\Jord(n)$ for $n\le 10$, which is as follows (if one denotes by $V_\lambda$, for $\lambda$ a partition of $n$, the irreducible $S_n$-module corresponding to $\lambda$).
\begin{gather*}
\Jord(1)\cong V_1,\\
\Jord(2)\cong V_2\\
\Jord(3)\cong V_{2,1}\oplus V_{3}\\
\Jord(4)\cong V_{2,1^2}\oplus V_{2^2}^2\oplus V_{3,1}\oplus V_{4}\\
\Jord(5)\cong V_{2,1^3}\oplus V_{2^2,1}^3\oplus V_{3,1^2}^2\oplus V_{3,2}^3\oplus V_{4,1}^2\oplus V_{5}\\
\Jord(6)\cong V_{2,1^4}\oplus V_{2^2,1^2}^3\oplus V_{2^3}^4\oplus V_{3,1^3}^4\oplus V_{3,2,1}^8\oplus V_{3^2}\oplus V_{4,1^2}^4\oplus V_{4,2}^6\oplus V_{5,1}^2\oplus V_6
\end{gather*}
\begin{multline*}
\Jord(7)\cong V_{2,1^5}\oplus V_{2^2,1^3}^4\oplus V_{2^3,1}^7\oplus V_{3,1^4}^5\oplus V_{3,2,1^2}^{16}\oplus V_{3,2^2}^{12}\\ 
\oplus V_{3^2,1}^9\oplus V_{4,1^3}^8\oplus V_{4,2,1}^{18}\oplus V_{4,3}^{7}\oplus V_{5,1^2}^6\oplus V_{5,2}^8\oplus V_{6,1}^3\oplus V_7
\end{multline*}
\begin{multline*}
\Jord(8)\cong V_{2,1^6}\oplus V_{2^2,1^4}^6\oplus V_{2^3,1^2}^{11}\oplus V_{2^4}^{10}\oplus V_{3,1^5}^5\oplus V_{3,2,1^3}^{26} 
\oplus V_{3,2^2,1}^{34}\\ \oplus V_{3^2,1^2}^{30}\oplus V_{3^2,2}^{19} \oplus V_{4,1^4}^{14}\oplus V_{4,2,1^2}^{41}\oplus V_{4,2^2}^{32}
\oplus V_{4,3,1}^{34}\oplus V_{4^2}^{10}\oplus V_{5,1^3}^{16}\\ \oplus V_{5,2,1}^{32}\oplus V_{5,3}^{12}\oplus V_{6,1^2}^9\oplus V_{6,2}^{12}\oplus V_{7,1}^3\oplus V_8
\end{multline*}

\begin{multline*}
\Jord(9)\cong V_{2,1^7}\oplus V_{2^2,1^5}^7\oplus V_{2^3,1^3}^{18}\oplus V_{2^4,1}^{22}\oplus V_{3,1^6}^6\oplus V_{3,2,1^4}^{38}\oplus V_{3,2^2,1^2}^{74} \oplus V_{3,2^3}^{44}\\ \oplus V_{3^2,1^3}^{58}\oplus V_{3^2,2,1}^{85}\oplus V_{3^3}^{20}\oplus V_{4,1^5}^{20}\oplus V_{4,2,1^3}^{84}\oplus V_{4,2^2,1}^{109}\oplus V_{4,3,1^2}^{107}\oplus V_{4,3,2}^{86}\oplus V_{4^2,1}^{44} \oplus V_{5,1^4}^{31}\\ \oplus V_{5,2,1^2}^{91}\oplus V_{5,2^2}^{64}\oplus V_{5,3,1}^{78}\oplus V_{5,4}^{22}\oplus V_{6,1^3}^{25}\oplus V_{6,2,1}^{53}\oplus V_{6,3}^{24}\oplus V_{7,1^2}^{12}\oplus V_{7,2}^{15}\oplus V_{8,1}^{4}\oplus V_9
\end{multline*}

\begin{multline*}
\Jord(10)\cong V_{2,1^8}\oplus V_{2^2,1^6}^7\oplus V_{2^3,1^4}^{26}\oplus V_{2^4,1^2}^{38}\oplus V_{2^5}^{26}\oplus V_{3,1^7}^8\oplus V_{3,2,1^5}^{53}\\ \oplus V_{3,2^2,1^3}^{139} \oplus V_{3,2^3,1}^{144} \oplus V_{3^2,1^4}^{93}\oplus V_{3^2,2,1^2}^{226}\oplus V_{3^2,2^2}^{122}\oplus V_{3^3,1}^{114}\oplus V_{4,1^{6}}^{26}\oplus V_{4,2,1^4}^{151} \oplus V_{4,2^2,1^2}^{272}\\ \oplus V_{4,2^3}^{162} \oplus V_{4,3,1^3}^{257}\oplus V_{4,3,2,1}^{394} \oplus V_{4,3^2}^{105}\oplus V_{4^2,1^2}^{143}\oplus V_{4^2,2}^{138}\oplus V_{5,1^5}^{50}\oplus V_{5,2,1^3}^{212}\oplus V_{5,2^2,1}^{263}\\ \oplus V_{5,3,1^2}^{289} \oplus V_{5,3,2}^{224} \oplus V_{5,4,1}^{144}\oplus V_{5,5}^{16}\oplus V_{6,1^4}^{58}\oplus V_{6,2,1^2}^{168}\oplus V_{6,2^2}^{120}\oplus V_{6,3,1}^{155}\oplus V_{6,4}^{50}\\ \oplus V_{7,1^3}^{40}\oplus V_{7,2,1}^{80}\oplus V_{7,3}^{35}\oplus V_{8,1^2}^{16}\oplus V_{8,2}^{20}\oplus V_{9,1}^{4}\oplus V_{10}
\end{multline*}
       
Let us outline the computation we performed to verify that these are indeed the decomposition of irreducibles of the corresponding components of the Jordan operad. Using Gröbner bases for operads \cite{OpGb}, we checked that the dimension of $\Jord(n)$ for $n\le 8$ coincides with the predicted values
 \[
 1, 1, 3, 11, 55, 330, 2345, 19089      
 \]
that one obtains by adding dimensions of irreducible $S_n$-modules with the multiplicities displayed above. Furthermore, we confirmed that answer by a computation using the \texttt{albert} program \cite{10.1145/190347.190358}; using that program, it is also possible to determine dimensions of multihomogeneous components of low degrees for Jordan algebras with at most $8$ generators, which gives the characters of the $GL(V)$-module $\Jord(V)_n$ for $n\le 8$, and hence, via the Schur--Weyl duality \cite{MR1153249}, the $S_n$-module structure of $\Jord(n)$ for such $n$, which coincides with the predictions above. However, going beyond degree $8$ in either of these programs does not seem feasible, since the above formulas suggest, in particular, that $\dim\Jord(9)=175203$ and $\dim\Jord(10)=1785840$, so the corresponding computations do not appear practical; additionally, converting a Gröbner basis calculation into the $S_n$-module decomposition would require another heavy computation that we did not attempt even for $n=8$. Thus, we needed a different strategy for $\Jord(9)$ and $\Jord(10)$, which we shall now describe.

The key idea is to directly determine the $S_n$-module $\Jord(n)$ by breaking everything into isotypic components at a very early stage of the computation.  This leads to a very drastic simplification, since the dimensions of vector spaces that we would want to compute will be noticeably lower. 
The idea of using representation theory for studying multilinear identities goes back to Specht \cite{MR35274} in the case of associative algebras and to Malcev \cite{MR33280} in case of possibly nonassociative algebras. Concrete algorithmic methods applying representation theory of symmetric group to study multilinear identities were first developed by the second author \cite{MR435159,MR463251}, and the fast method of computation of matrices of irreducible representations of $S_n$ is due to Clifton \cite{MR624907,MR2630923}. The reader is also  invited to consult the recent survey \cite{MR3583300} for a detailed description of the method which we now briefly outline.
Basically, if one has a multilinear identity $f$ of degree $n$, a naive method of determining its consequences of degree $n+1$ would mean forming $(n+2)!$ identities, and then determining the dimension of vector space they generate inside the vector space spanned by  all nonassociative monomials. Here the number $(n+2)!=(n+2)(n+1)!$ corresponds to the fact there are $n$ identities of degree $n+1$ where one of the arguments $x_i$ is replaced by $x_ix_{n+1}$,\footnote{From the operad theory point of view, it would be a bit more natural to consider the $n$ ``completely nonsymmetric consequences'' $f(x_1,\ldots,x_ix_{i+1},x_{i+2},\ldots,x_{n+1})$.} and two identities $x_{n+1}f$ and $fx_{n+1}$, and then each of these $n+2$ identities can be acted upon by $(n+1)!$ elements of $S_{n+1}$). The dimension of the vector space spanned by all nonassociative monomials is $c_{n+1}(n+1)!$, where $c_n=\frac1n\binom{2n-2}{n-1}$ is the number of valid bracketings for a nonassociative monomial of degree $n$, known as the Catalan number. Needless to say, this naive approach becomes vastly unpractical very quickly, as factorials grow too fast. Breaking everything into isotypic components means that one can work with vector spaces of multiplicities of irreducible $S_{n+1}$-modules, so the above strategy would be replaced by looking, for every irreducible representation $V_\lambda$ of $S_{n+1}$ with $\dim V_\lambda=d_\lambda$, at $(n+2)d_\lambda$ elements inside a vector space of dimension $c_{n+1}d_\lambda$, the multiplicity space of $V_\lambda$ in the vector space spanned by  all nonassociative monomials. This allows one to advance much further.

In our particular case of Jordan algebras, there are several other simplifications to consider. First of all, Jordan algebras are commutative, and therefore the Catalan numbers will be replaced by the so called Wedderburn--Etherington numbers \cite{oeis-eth} that enumerate the commutative association types; up to degree $10$ these are
 \[
1, 1, 1, 2, 3, 6, 11, 23, 46, 98.     
 \]
Moreover, for a Jordan product, one can use the well known identity (see, e. g., \cite[Lemma B.2.1]{MR2014924})
 \[
L_{(ac)b}=-L_aL_bL_c-L_cL_bL_a+L_{ac}L_b+L_{ab}L_c+L_{bc}L_a
 \]
that allows one to reduce the number of associations even further; this formula means that we may focus only on association types
of the form
 \[
(\ldots((x_1 x_2)x_3)\ldots)x_n ,     
 \]
where each $x_i$ is of degree $1$ or $2$ (a generator or a product of generators). This means that the number of association
types is the appropriate Fibonacci number $f_n$, and up to degree $10$ these are
 \[
1, 1, 1, 2, 3, 5, 8, 13, 21, 34.     
 \]
Thus, in our case, to determine all Jordan identities in degree $n$, we determine the degree $n$ nonsymmetric consequences of the Jordan identity (that is, consequences that only use substitutions, but not the symmetric group action), whose number we denote by $j_{n}$, and then work, for each irreducible representation $V_\lambda$ of $S_{n}$ with a matrix with $f_{n}d_\lambda$ rows (corresponding to the vector space where our identities are expanded, which is the tensor product of the vector space of different association types and the multiplicity space of $V_\lambda$ in $\k S_n$) and $j_{n}d_\lambda$ columns (corresponding to the tensor product of the vector space of degree $n$ nonsymmetric consequences and the multiplicity space of $V_\lambda$ in $\k S_n$). If $r_\lambda$ is the rank of that matrix, then 
 \[
f_nd_\lambda-r_\lambda
 \] 
is the multiplicity of $V_\lambda$ in $\Jord(n)$. For $n=9$ and $n=10$, this allowed us to confirm that the $S_n$-module structure of $\Jord(n)$ coincides with the prediction of Conjecture \ref{conj:conj2}.  

To conclude the proof, it remains to note that for any vector space $V$, we have 
 \[
\Jord(V)_n=\Jord(n)\otimes_{\k S_n} V^{\otimes n},     
 \]
proving the claim.
\end{proof}

\begin{remark}
The computation of the actual module structure for $n=11$, even if optimized using the representation theory of symmetric groups, is a bit too heavy, so we did not attempt it at the moment. However, using \texttt{SageMath}, we computed the prediction of Conjecture \ref{conj:conj2} about the $S_n$-module structure of $\Jord(n)$ for $n\le 32$, and the corresponding module is effective (all multiplicities of irreducibles are non-negative), which we regarded as further very strong evidence in favor of Conjecture \ref{conj:conj1}. 
\end{remark}

\subsection{Computations for three even generators}
Using the \texttt{albert} program, we computed some multigraded dimensions of the free Jordan algebra on three generators beyond degree $10$.

Namely, for degrees $11$ and $12$ and the multidegrees with all components strictly positive (that is, not contained in two-generated subalgebras), we have the following dimensions:
 \[
\begin{tabular}{|c|c|}
\hline
multidegree & dimension \\
\hline
(9,1,1)& 55\\
(8,2,1)& 250\\
(7,3,1)& 660\\
(6,4,1)& 1160\\
(5,5,1)& 1386\\
(7,2,2)& 1000\\
(6,3,2)& 2326\\
(5,4,2)& 3493\\
(5,3,3)& 4651\\
(4,4,3)& 5835\\
\hline
\end{tabular}\qquad 
\begin{tabular}{|c|c|}
\hline
multidegree & dimension \\
\hline
(10,1,1) &66\\
(9,2,1)& 330\\
(8,3,1)& 990\\
(7,4,1)& 1980\\
(6,5,1)& 2772\\
(8,2,2)& 1500\\
(7,3,2)& 3969\\
(6,4,2)& 6982\\
(5,5,2)& 8347\\
(6,3,3)& 9291\\
(5,4,3)& 13961\\
(4,4,4)& 17520\\
\hline
\end{tabular}
 \]
All of them agree with Conjecture \ref{conj:conj2}.

\subsection{The case of two even generators}\label{sec:2gen}

To test the conjecture for two even generators, one may use the results of Cohn \cite{MR60496} and Shirshov \cite{MR75936} that together ensure that the free Jordan algebra $\Jord(x_1,x_2)$ is the free \emph{special} Jordan algebra $\Jord(x_1,x_2)$: it embeds into the Jordan algebra associated to the free associative algebra on two generators; moreover, the image of that embedding can be identified as the space of ``reversible'' elements, that is, the space of invariants the only antiautomorphism of the free algebra that is identical on the space of generators. This dimension is given by the OEIS sequence \cite{oeis-sym} with the first $20$ terms
\begin{multline*}
2, 3, 6, 10, 20, 36, 72, 136, 272, 528, 1056, 2080, 4160, 8256, 16512,\\
 32896, 65792, 131328, 262656, 524800\ldots     
\end{multline*} 
Moreover, it is explained in \cite[Sec.~2.7]{MR4235202} that 
 \[
\dim\calB(\Jord(x_1,x_2))_n=2^n-\dim\Jord(x_1,x_2)_n-N_{\mathrm{necklace}}(n)+N_{\mathrm{bracelet}}(n),     
 \]
where $N_{\mathrm{necklace}}(n)=A000031(n)$ and $N_{\mathrm{bracelet}}(n)=A000029(n)$ are the numbers of necklaces and bracelets made of $n$ beads of two colours, given by the OEIS sequences \cite{oeis-nec} and \cite{oeis-bra} respectively, giving the sequence of dimensions with the first $20$ terms
\begin{multline*}
0, 1, 2, 6, 12, 27, 54, 114, 226, 466, 930, 1888, 3780, 7633, 15288,\\ 30774, 61680, 123899, 248346, 498300,\ldots     
\end{multline*} 

It is checked in \cite[Sec.~2.7]{MR4235202} that Conjecture \ref{conj:conj2} predicts the correct values of $\dim\Jord(x_1,x_2)_n$ for $n\le 15$. Because of that, the following result came as a complete surprise. 

\begin{proposition}\label{prop:Jord20}
Conjecture \ref{conj:conj1} gives correct predictions for $\dim\Jord(x_1,x_2)_n$ for $n\le 18$, but an incorrect prediction for $\dim\Jord(x_1,x_2)_{19}$. Thus, it fails for the free Jordan algebra $\Jord(x_1,x_2)$.
\end{proposition}

\begin{proof}
Recall that the weakest conjecture of \cite{MR4235202} states that for a given $p$ the dimensions $a_n(p):=\dim\Jord(x_1,\ldots,x_p)_n$ form the unique sequence satisfying
 \[
\mathrm{Res}_{t=0}(1-pz-t+pzt^{-1})\prod_{n\ge 1}(1-z^n(t+t^{-1})+z^{2n})^{a_n(p)}\,dt=0.     
 \]
(Here, as usual, the residue at $t=0$ means the coefficient of $t^{-1}$.)  
This means that verifying that for all $n\le N$ the number $a_n(2)$ defined by this property coincides with $A005418(n)$ 
can be done very directly.
We computed in \texttt{SageMath} \cite{sagemath} and independently in \texttt{PARI/GP} \cite{PARI2} the coefficients of $z^k$ in 
\begin{equation}\label{eq:residue}
(1-2z-t+2zt^{-1})\prod_{n=1}^{19}(1-z^n(t+t^{-1})+z^{2n})^{A005418(n)}    
\end{equation}
for $k\le 19$. To organize this computation, it is useful to note that, since we only care about coefficients in low degrees, we may work in the quotient ring $\mathbb{Q}[t,t^{-1}][z]/(z^{20})$, where the huge powers in \eqref{eq:residue} do not have any drastic impact on the complexity of the computation. For instance, one may write
 \[
\prod_{n=10}^{19}(1-z^n(t+t^{-1})+z^{2n})^{A005418(n)}=1-(t+t^{-1})\sum_{n=10}^{19}A005418(n)z^n.   
 \]
This and similar simplifications make it clear that the relevant computation can be done very efficiently, and is even susceptible to a step-by-step human verification.

It turns out that the coefficients of $z^k$ in \eqref{eq:residue} have the residue $0$ for $k\le 18$, but the coefficient for $k=19$ is 
\begin{multline*}
-1218t^{9} + 45184t^{8} - 472252t^{7} + 2389852t^{6} - 7383950t^{5} + 15783200t^{4} - 24906124t^{3} \\
+ 29605472t^{2} - 25748624t^{1} + 13996812 + 2t^{-1} - 9637460t^{-2} + 11988306t^{-3}  \\
- 8946852t^{-4} + 4523548t^{-5} - 1498732t^{-6} + 286354t^{-7} - 23908^t{-9} + 390t^{-9},     
\end{multline*}
and in particular its residue is equal to $2$. This implies that the prediction for $\dim\Jord(x_1,x_2)_k$ is correct for $k\le 18$, but fails for $k=19$. Thus, Conjecture \ref{conj:conj2} does not hold, and hence Conjecture \ref{conj:conj1} does not hold either.   
\end{proof}

\begin{remark}
In fact, one can show that the prediction of Conjecture \ref{conj:conj2} for the dimension of $\Jord(x,y)_{19}$ is $262658$, and not the correct value $262656$ listed above. One can also compute the prediction of Conjecture \ref{conj:conj2} for dimensions of $\calB(\Jord(x,y))$, and find that in degree $20$ it gives $498303$, and not the correct value $498300$ listed above.
\end{remark}

The result we proved immediately implies the following.

\begin{theorem}
Conjecture \ref{conj:conj1} fails for the free Jordan superalgebra 
 \[
\Jord(x_1,\ldots,x_m\mid y_1,\ldots,y_n)   
 \]
whenever $m>1$.
\end{theorem}

\begin{proof}
The free Jordan superalgebra $\Jord(x_1,\ldots,x_m\mid y_1,\ldots,y_n)$ is $\mathbb{N}^{m+n}$-graded, and this grading induces a grading on the Lie superalgebra 
 \[
\mathrm{TAG}(\Jord(x_1,\ldots,x_m\mid y_1,\ldots,y_n)),     
 \]
as well as on its homology. Considering the graded components supported on $\mathbb{N}^I\subset \mathbb{N}^{m+n}$ for various subsets $I$ of the set of variables, we see that our assertion follows from the particular case $(m\mid n)=(2\mid 0)$, which is our Proposition \ref{prop:Jord20}.
\end{proof}

\begin{remark}
Our results do not shed light on the question as to whether Conjectures \ref{conj:conj1} and \ref{conj:conj2} hold for the free Jordan superalgebras $\Jord(x_1,\ldots,x_m\mid y_1,\ldots,y_n)$ with $m\le 1$. It would be interesting to determine if that is indeed the case, and in particular to describe the free two-generated Jordan superalgebras $\Jord(x\mid y)$ and $\Jord(\varnothing\mid y_1,y_2)$. These appear to be more complicated objects than the free algebra $\Jord(x_1,x_2\mid\varnothing)$. The reason for it is that the latter Jordan algebra is special (that is, isomorphic to a Jordan subalgebra of an associative algebra) while the two former are not (as the authors learned from I. Shestakov, there exist exceptional Jordan superalgebras generated by two elements, which can be chosen to be both odd, or to be odd and even).
\end{remark}

\subsection{Inner derivations of free Jordan algebras}
We conclude with indicating one question concerning free Jordan algebras that has a chance for a positive answer. Specifically, in \cite[Lemma 12]{MR4235202}, it is established that Conjecture \ref{conj:conj1} would imply that we have 
$\calB(\Jord(V))\cong\mathrm{Inner}(\Jord(V))$ for any choice of a vector space $V$. Additionally, in \cite[Corollary~2]{MR4235202}, it is shown that the same assertion holds for free special Jordan algebras (that is, Jordan subalgebras of free associative algebras equipped with the product $ab=\frac12(a\cdot b+b\cdot a)$): 
 \[
\calB(\mathrm{SJord}(V))\cong \mathrm{Inner}(\mathrm{SJord}(V)).
 \] 
As we mentioned in Section \ref{sec:2gen}, for $\dim(V)=2$ we have $\Jord(V)\cong \mathrm{SJord}(V)$, which means that in that case we have $\calB(\Jord(V))\cong\mathrm{Inner}(\Jord(V))$, even though \cite[Lemma 12]{MR4235202} is not applicable. This raises a natural question whether we always have 
 \[
\calB(\Jord(V))\cong\mathrm{Inner}(\Jord(V))     
 \]
even though it cannot be deduced from the (false) Conjecture \ref{conj:conj1}.

\printbibliography

\end{document}